\documentclass[10pt]{amsart}
\usepackage{amssymb}
\usepackage{bm}
\usepackage{graphicx}
\usepackage[centertags]{amsmath}
\usepackage{amsfonts}
\usepackage{amsthm}
\usepackage{graphicx}
\newtheorem{thm}{Theorem}
\newtheorem{cor}[thm]{Corollary}
\newtheorem{lem}[thm]{Lemma}

\newtheorem{prop}[thm]{Proposition}

\newtheorem{fact}[thm]{Fact}

\newtheorem{defn}[thm]{Definition}
\theoremstyle{definition}

\newcommand{\nn}{\mathbb{N}}
\newcommand{\ee}{\varepsilon}
\newcommand{\con}{\smallfrown}
\newcommand{\meg}{\geqslant}
\newcommand{\mik}{\leqslant}

\newcommand{\supp}{\mathrm{supp}}

\begin{document}

\title[Finite version of Gowers' $c_0$ theorem]{Primitive recursive bounds for the finite version of Gowers' $c_0$ theorem}

\author{Konstantinos Tyros}

\address{Mathematics Institute, University of Warwick, Coventry, CV4 7AL, UK}
\email{k.tyros@warwick.ac.uk}

\thanks{Supported by ERC grant 306493}
\maketitle


\begin{abstract}
We provide primitive recursive bounds for the finite version of Gowers' $c_0$ theorem for both
the positive and the general case. We also provide multidimensional versions of these results.
\end{abstract}


\section{Introduction}
In 1992 W. T. Gowers (see \cite{G1}) obtained a stability result for real valued Lipschitz functions defined on the unit sphere of $c_0$. This result is actually a consequence of a deep infinite dimensional Ramsey type result (see \cite{BL,G2,T} as well as \cite{K} for an elegant proof).
Our goal in this paper is to provide primitive recursive bounds for the finite version of this
Ramsey type result.

To state our results we need first to introduce some pieces of notation. By $\nn$ we denote the set of all non-negative integers.
Let $n,k$ be positive integers. By $X_k(n)$ we denote the set of all functions having domain the
set $\{0,...,n-1\}$, range the set $\{0,...,k\}$ and achieving the value $k$, i.e.,
\[X_k(n)=\big\{f:\{0,...,n-1\}\to\{0,...,k\}\;\text{such that}\;f(i)=k
\;\text{for some}\;0\mik i< n\big\}.\]
By $X_{[k]}(n)$ we denote the set of all functions having domain the
set $\{0,...,n-1\}$ and range the set $\{0,...,k\}$.
The set $X_k(n)$ discretizes  the positive cone of the unit sphere of $\ell_\infty^n$, while $X_{[k]}(n)$ discretizes the
positive cone of the unit ball of $\ell_\infty^n$.
The scalar multiplication is captured by the following map.
We define $T:X_{[k]}(n)\to X_{[k]}(n)$ as follows. For every $f\in X_{[k]}(n)$
and $i\in\{0,...,n-1\}$ we set
\[T(f)(i)=\max(0,f(i)-1).\]
For every $f$ in $X_{[k]}(n)$, by $\supp(f)$, we denote the set of all $i$ in $\{0,...,n-1\}$ such that $f(i)\neq0$. A sequence $\mathbf{F}=(f_i)_{i=0}^{m-1}$ in $X_{[k]}(n)$ is called block of length $m$ if $(\supp(f_i))_{i=0}^{m-1}$ forms a
block sequence of nonempty finite subsets of $\nn$, that is, $\max\supp(f_i)<\min\supp(f_{i+1})$ for all $0\mik i<m-1$. For a block sequence $\mathbf{F}=(f_i)_{i=0}^{m-1}$ in $X_{[k]}(n)$, we define the positive subspace generated by $\mathbf{F}$ to be
\[\begin{split}\langle \mathbf{F} \rangle_k=\Big\{\sum_{i=1}^{\ell}T^{\ee_i}(f_{j_i}):\;&
\ell\mik m,\; 0\mik j_1<...<j_\ell<m,\; \\
&\ee_1,...,\ee_\ell\in\{0,...,k-1\}\;\text{and}\;\min_{1\mik i\mik\ell}\ee_i=0\Big\}.
\end{split}\]
The first result in this paper is the following.

\begin{thm}
  \label{positive_Gowers_fin}
  For every triple of positive integers $k,m,r$, there exists a positive integer $n_0$ satisfying the
  following property. For every integer $n\meg n_0$ and every coloring of the set $X_k(n)$ with $r$ colors, there exists a block sequence $\mathbf{F}$ in $X_k(n)$ of length $m$ such that the set $\langle \mathbf{F} \rangle_k$ is monochromatic.
  We denote the least $n_0$ satisfying the above property by $ \mathrm{G}(k,m,r)$.

  Moreover, the numbers $\mathrm{G}(k,m,r)$ are upper bounded by a primitive recursive function belonging to the class $\mathcal{E}^7$ of Grzegorczyk's hierarchy.
\end{thm}

For a detailed exposition on the Grzegorczyk's classes we
 refer the reader to \cite{R}. A similar to Theorem \ref{positive_Gowers_fin} result (not obtaining, however, primitive recursive bounds)
have been recently proved in \cite{O}. The proof of Theorem \ref{positive_Gowers_fin}
is inspired by Shelah's proof for the Graham--Rothschild theorem (see \cite[Theorem 2.2]{Sh}),
though there are some novel arguments that do not appear in Shelah' proof.
Specifically, given a coloring of the set $X_k(n)$, the goal is, by passing to a subspace, to canonize this coloring and make it
insensitive in the following sense. Every two functions of the form $f$ and $f+T^{(k-1)}(f')$ have the same color, where $f$ and $f'$ are disjointly supported.
In order to achieve this insensitivity of a given coloring, we consider the type of a function, a notion
introduced in Section \ref{section_typ_ins}, that takes into account the nonzero values of a function with no repetition. Using the finite version of the Milliken--Taylor theorem, we first make the coloring to depend only on the type of a function. An appropriate choice of functions, that eliminates several types, generates a subspace on which the given coloring is insensitive.

Our second result involves functions taking also negative values. To state it, we need to introduce some additional notation.
Let $n,k$ be positive integers. By $X_{\pm k}(n)$ we denote the set of all functions having domain the
set $\{0,...,n-1\}$, range the set $\{-k,...,k\}$ and achieving the value $k$ or $-k$, i.e.,
\[X_{\pm k}(n)=\big\{f:\{0,...,n-1\}\to\{-k,...,k\}\;\text{such that}\;|f(i)|=k
\;\text{for some}\;0\mik i< n\big\}.\]
By $X_{[\pm k]}(n)$ we denote the set of all functions having domain the
set $\{0,...,n-1\}$ and range the set $\{-k,...,k\}$. Moreover, we extend the map $T:X_{[\pm k]}(n)\to X_{[\pm k]}(n)$ as follows. We set
\[
T(f)(i)=\left\{ \begin{array} {l} f(i)-1,\;\;f(i)>0\\
                               0,\;\;\;\;\;\;\;\;\;\;\;\;\;f(i)=0\\
                               f(i)+1, \;\;f(i)<0     \end{array}  \right.
\]
for all $f$ in $X_{[\pm k]}(n)$ and $i$ in $\{0,...,n-1\}$.
For a block sequence $\mathbf{F}=(f_i)_{i=0}^{m-1}$ in $X_{\pm k}(n)$, i.e., $(\supp(f_i))_{i=0}^{m-1}$
is a block sequence, we extend the notion of the positive subspace. In particular, we define the subspace generated by
$\mathbf{F}=(f_i)_{i=0}^{m-1}$ to be
\[\begin{split}\langle \mathbf{F} \rangle_{\pm k}=\Big\{\sum_{i=1}^{\ell}\pm T^{\ee_i}(f_{j_i}):\;&
\ell\mik m,\; 0\mik j_1<...<j_\ell<m,\; \\
&\ee_1,...,\ee_\ell\in\{0,...,k-1\}\;\text{and}\;\min_{1\mik i\mik\ell}\ee_i=0\Big\}.
\end{split}\]
On $X_{\pm k}(n)$ we consider the supremum metric, denoted by $\rho_\infty$, and defined as usual
\[\rho_\infty(f,g)=\max_{0\mik i<n}|f(i)-g(i)|\]
for all $f,g$ in $X_{\pm k}(n)$.
Given a finite coloring  $c:X_{\pm k}(n)\to\{1,...,r\}$, we say that a subset
$A$ of $X_{\pm k}(n)$ is \textit{approximately monochromatic} if there exists some $i_0$ in $\{1,...,r\}$
such that for every $f$ in $A$ there exists an $f'$ in $X_{\pm k}(n)$ with $c(f')=i_0$ and $\rho_\infty(f,f')\mik1$.

\begin{thm}
  \label{full_Gowers_fin}
  For every triple of positive integers $k,m,r$, there exists a positive integer $n_0$ satisfying the
  following property. For every integer $n\meg n_0$ and every coloring of the set $X_{\pm k}(n)$ with $r$ colors, there exists a block sequence $\mathbf{F}$ in $X_{\pm k}(n)$ of length $m$ such that the set $\langle \mathbf{F} \rangle_{\pm k}$ is approximately monochromatic.
  We denote the least $n_0$ satisfying the above property by $\mathrm{G}_\pm(k,m,r)$.

  Moreover, the numbers $\mathrm{G}_\pm(k,m,r)$ are upper bounded by a primitive recursive function belonging to the class $\mathcal{E}^7$ of Grzegorczyk's hierarchy.
\end{thm}

The reduction of Theorem \ref{full_Gowers_fin} to Theorem \ref{positive_Gowers_fin}
essentially relies on the choice of some appropriate functions in $X_{\pm k}(n)$.
This choice has been inspired by the approach in \cite{K} that makes use of $-T$ instead
of $T$ for the proof of the general case.

Finally we consider multidimensional versions of the Theorems \ref{positive_Gowers_fin} and \ref{full_Gowers_fin}, which are presented in Section \ref{section_mult_versions}.


\section{Background Material}
In this section we gather some background material needed in this paper.
\subsection{Grzegorczyk's hierarchy} In order to analyze the rate of growth of the several bounds resulting from our arguments, we will make use of Grzegorczyk's hierarchy of the primitive recursive functions. In this section we overview the basic facts related to this hierarchy.

By the term \textit{number theoretic function} we mean a function of the form $f:\nn^k\to\nn$, where $k$ is a positive integer. The basic examples of number theoretic functions are the constant zero function $z:\nn\to\nn$, the successor function $S:\nn\to\nn$ defined by the rule $S(n)=n+1$  and the projection functions $P^k_j:\nn^k\to\nn$, where $j$ and $k$ are positive integers with $j\mik k$, defined by the rule $P^k_j(n_1,...,n_k)=n_j$.

 Let $h$ be a number theoretic function of arity $k$ and $g_1,...,g_k$
number theoretic functions all of arity $\ell$. We define the \emph{composition} of $h$ with $g_1,...,g_k$ to be the function $f$ of arity $\ell$ defined by the rule
\[f(n_1,...,n_\ell)=h\big(g_1(n_1,...,n_\ell),...,g_k(n_1,...,n_\ell)\big).\]

 The notion of primitive recursion possesses a central role in this setting. Let $g$, $h$ and $f$ be number theoretic functions of arities $k$, $k + 2$ and $k + 1$ respectively. We say that $f$ is defined by primitive recursion from $g$ and $h$ if for every $n,n_1,...,n_k$ in $\nn$ we have
 \[
\left\{ \begin{array} {l} f(0,n_1,...,n_k)=g(n_1,...,n_k)\\
                          f(n+1,n_1,...,n_k)=h(f(n,n_1,...,n_k),n,n_1,...,n_k).\end{array}  \right.
\]

 \begin{defn}
   The class $\mathcal{E}$ of the primitive recursive functions is the smallest class of number theoretic functions that contains the constant zero function, the successor
function and the projection functions, and is closed under composition and primitive
recursion.
 \end{defn}

 We consider the following sequence $(E_q)_q$ of number theoretic functions. The functions $E_0:\nn^2\to\nn$ and $E_1:\nn\to\nn$ are defined by the rules $E_0(n,m)=n+m$
 and $E_1(n)=n^2+2$, while for every positive integer $q$ the function $E_{q+1}:\nn\to\nn$ is defined recursively by the rule
 \[
\left\{ \begin{array} {l} E_{q+1}(0)=2\\
                          E_{q+1}(n+1)=E_q(E_{q+1}(n)).\end{array}  \right.
\]
 Clearly, each $E_q$ is primitive recursive and for every positive integer $q$ the function $E_q$ is increasing.

 Moreover, we say that a class $\mathcal{C}$ of number theoretic functions is closed under
 limited primitive recursion, if for every number theoretic functions $f$, $g$, $h$ and $j$ such that
 \begin{enumerate}
   \item[(i)] $g,h$ and $j$ belong to $\mathcal{C}$,
   \item[(ii)] $f$ is defined by primitive recursion from $g$ and $h$,
   \item[(iii)] $f$ and $j$ have the same arity and
   \item[(iv)] $f$ is pointwise bounded by $j$,
 \end{enumerate}
  we have that $f$ belongs to $\mathcal{C}$.

 \begin{defn}
The Grzegorczyk's class $\mathcal{E}^0$ is the smallest
class of number theoretic functions that contains the constant zero
function, the successor function and the projection functions, and is closed under
composition and limited primitive recursion.

   For every positive integer $q$ the Grzegorczyk's class $\mathcal{E}^q$ is the smallest
class of number theoretic functions that contains the function $E_{q-1}$, the constant zero
function, the successor function and the projection functions, and is closed under
composition and limited primitive recursion.
 \end{defn}

 Some of the basic properties of the hierarchy $(\mathcal{E}^q)_q$ are isolated in the following
 proposition (see \cite[Section 2.2]{R} for further details).

 \begin{prop}
 \label{prim_rec_properties}
   The following hold.
   \begin{enumerate}
     \item[(i)] The hierarchy $(\mathcal{E}^q)_q$ is strictly increasing and $\mathcal{E}=\bigcup_q\mathcal{E}^q$.
     \item[(ii)] If $g, h$ belong to $\mathcal{E}^q$ for some $q$ in $\nn$ and $f$ is defined by primitive recursion from $g$ and $h$, then $f$ belongs to $\mathcal{E}^{q+1}$.
     \item[(iii)] For every integer $q\meg2$ and every $f$ in $\mathcal{E}^q$ there exists an $m$ in $\nn$ such that for every $n_1,...,n_k$ in $\nn$ we have
$f(n_1,...,n_k) \mik E_{q-1}^{(m)}\big(\max\{n_1,...,n_k\}\big)$, where $k$ is the arity of f.
   \end{enumerate}
 \end{prop}

 By conclusion (iii) of the above proposition we have, in particular, that every number theoretic function belonging to some class $\mathcal{E}^q$ is dominated by an increasing unary function in
 the same class $\mathcal{E}^q$. More
generally, we have the following corollary.

\begin{cor}
  For every $q$ in $\nn$ and every $f$ in $\mathcal{E}^q$ of arity $k$ there exists
  a function $F$ in $\mathcal{E}^q$ of arity $k$ that pointwise dominates $f$ and satisfies
  \begin{equation}
    \label{eq_mon}
    F(n_1,...,n_k)\mik F(m_1,...,m_k)
  \end{equation}
  for every choice of $n_1,...,n_k,m_1,...,m_k$ with $n_i\mik m_i$ for all $i=1,...,k$.
\end{cor}

Thus, by the above corollary, we may assume that all primitive recursive functions
we are dealing with in the present paper satisfy the monotonicity property described in \eqref{eq_mon}.

\subsection{The Milliken--Taylor Theorem}
The proofs of Theorems \ref{positive_Gowers_fin} and \ref{full_Gowers_fin}
make use of the finite version of the Milliken--Taylor theorem \cite{M, Tay1}.
To state it we need to introduce some pieces of
notation. Let $m,d$ be positive integers with $d\mik m$.  A finite sequence $\mathbf{s}=(s_i)_{i=0}^{m-1}$ of nonempty finite subsets of $\nn$ is called block if $\max s_i<\min s_{i+1}$
for all $0\mik i<m-1$.
For a block sequence $\mathbf{s}=(s_i)_{i=0}^{m-1}$ of nonempty finite subsets of $\nn$ we define the set of nonempty unions of $\mathbf{s}$ to be
\[\mathrm{NU}(\mathbf{s})=\Big\{\bigcup_{i\in t}s_i:t\;\text{is a nonempty subset of}\;\{0,...,m-1\}\Big\}.\]
We say that a block sequence $\mathbf{t}=(t_i)_{i=0}^{d-1}$ of nonempty finite subsets of $\nn$ is a block subsequence of $\mathbf{s}$ if $t_i\in\mathrm{NU}(\mathbf{s})$ for all $0\mik i<d$.
By $\mathrm{Block}^d(\mathbf{s})$ we denote the set of all block subsequences of $\mathbf{s}$ of length $d$. Moreover, for simplicity, by $\mathrm{Block}^d(m)$, we denote the set $\mathrm{Block}^d\big((\{i\})_{i=0}^{m-1}\big)$.
The finite version of the Milliken-Taylor theorem is stated as follows.
\begin{thm}
  \label{Mil_Tay_fin}
  For every triple $d,m,r$ of positive integers with $d\mik m$, there
exists a positive integer $n_0$ with the following property. For every finite block sequence $\mathbf{s}$ of nonempty
finite subsets of $\nn$
of length at least $n_0$ and every coloring of the set $\mathrm{Block}^d(\mathbf{s})$ with $r$ colors, there exists a block subsequence $\mathbf{t}$ of $\mathbf{s}$ of length $m$
such that the set $\mathrm{Block}^d(\mathbf{t})$ is monochromatic.
We denote the least $n_0$ satisfying the above property by $\mathrm{MT}(d,m,r)$.

  Moreover, the numbers $\mathrm{MT}(d,m,r)$ are upper bounded by a primitive recursive function belonging to the class $\mathcal{E}^6$ of Grzegorczyk's hierarchy.
\end{thm}

Note that the case ``$d=1$'' of Theorem \ref{Mil_Tay_fin} is the finite version of Hindman's theorem \cite{H}.
This finite version follows by the disjoint union theorem \cite{GR, Tay2} and Ramsey's theorem.
The bounds for the disjoint union theorem given in \cite{Tay2}, as well as, the bound for the Ramsey numbers given in \cite{ER} are in $\mathcal{E}^4$. Using these bounds, one can see that the
numbers $\mathrm{MT}(1,m,r)$ are upper bounded by a primitive recursive function belonging to the class $\mathcal{E}^4$ of Grzegorczyk's hierarchy.
The higher dimensional case of Theorem \ref{Mil_Tay_fin} (that is, the case $d\meg2$), follows
by a standard iteration argument similar to the one in Section \ref{section_mult_versions}
(see \cite{DK} for further details).


\section{Types and Insensitivity}\label{section_typ_ins}
Let us start with some notation. Let $d,n$ be positive integers with $d\mik n$ and $\mathbf{F}$ be a block sequence in $X_k(n)$ (resp. in $X_{\pm k}(n)$). We say that a block sequence $\mathbf{G}=(g_i)_{i=0}^{d-1}$ in  $X_k(n)$ (resp. in $X_{\pm k}(n)$) is a block subsequence of $\mathbf{F}$
if $g_i$ belongs to $\langle\mathbf{F}\rangle_k$ (resp. $g_i$ belongs to $\langle \mathbf{F} \rangle_{\pm k}$) for all $0\mik i<d$.
Moreover, by $\mathrm{Block}_k^d(\mathbf{F})$ (resp. $\mathrm{Block}_{\pm k}^d(\mathbf{F})$) we denote the set of all block subsequences of $\mathbf{F}$ of length $d$. For simplicity, by $\mathrm{Block}_k^d(n)$ (resp. $\mathrm{Block}_{\pm k}^d(n)$), we denote the set $\mathrm{Block}_k^d((k\cdot\chi^n_{\{i\}})_{i=0}^{n-1})$ (resp. $\mathrm{Block}_{\pm k}^d((k\cdot\chi^n_{\{i\}})_{i=0}^{n-1})$), where by $\chi^n_{A}$ we denote the characteristic function, defined on $\{0,...,n-1\}$, of a finite nonempty subset $A$ of $\nn$ with $\max A<n$. Finally, for every finite sequence $\mathbf{b}$ and every non-negative integer
$d$ less that or equal to the length of $\mathbf{b}$, by $\mathbf{b}|d$, we denote the initial segment of $\mathbf{b}$ of
length $d$.

The proof of Theorem \ref{positive_Gowers_fin} proceeds by induction on $k$.
Central role in the proof of the inductive step possesses
the notion of insensitivity, which we are about to define.
\begin{defn}
  Let $m,n,k$ be positive integers with $m\mik n$ and $\mathbf{F}\in \mathrm{Block}_{k}^{m}(n)$.
  We say that a coloring of the set $X_k(n)$ is insensitive over $\mathbf{F}$ if for
  every $f,f'$ in $\langle\mathbf{F}\rangle_k$
  disjointly supported, we have that $f$ and $f+T^{(k-1)}(f')$ have the same color.
\end{defn}

Let as point out that the notion of insensitivity is hereditary. In particular, we have the
following easy to observe fact.

\begin{fact}
  \label{fact_ins}
  Let $m,n,k$ be positive integers with $m\mik n$ and $\mathbf{F}\in \mathrm{Block}_{k}^{m}(n)$.
  Also let $c$ be a coloring of the set $X_k(n)$ and assume that $c$ is insensitive over $\mathbf{F}$.
  Then for every block subsequence $\mathbf{G}$ of $\mathbf{F}$, we have that $c$ is insensitive over $\mathbf{G}$.
\end{fact}

In order to carry out the inductive step of the proof of Theorem \ref{positive_Gowers_fin}
we need to make an arbitrary coloring of the set $X_k(n)$ insensitive over a long enough block
sequence $\mathbf{F}$ in $X_k(n)$. To this end, we will need the notion of a type.

Let $n,k$ be positive integers.
For every $d\mik n$, $\mathbf{s}=(s_i)_{i=0}^{d-1}$ in $\mathrm{Block}^d(n)$ and
$g\in X_{[\pm k]}(d)$, we set
\[\mathrm{map}(g,\mathbf{s})=\sum_{i=0}^{d-1}g(i)\chi^n_{s_i}\]
which is clearly an element of $X_{[\pm k]}(n)$.
Moreover, for every $d\mik n$, a type of length $d$ over $k$ (resp $\pm k$) is a function $\varphi$ in $X_k(d)$ (resp. $X_{\pm k}(d)$), such that $\varphi(i)\neq\varphi(i+1)$ for all $0\mik i<d-1$ and $\supp(\varphi)=\{0,...,d-1\}$.
By $|\varphi|$, we denote the length of a type $\varphi$.
Observe that for every function $f$ in $X_{\pm k}(n)$ there exist a
unique type of some length $d\mik n$, which we denote by $\mathrm{tp}(f)$, and a unique block sequence of nonempty finite subsets of
$\nn$ of length $d$, which we denote by $\mathrm{bsupp}(f)$, such that $f=\mathrm{map}(\mathrm{tp}(f),\mathrm{bsupp}(f))$.
Finally, for $d\mik n$ and $\mathbf{s}=(s_i)_{i=0}^{d-1}$
in $\mathrm{Block}^d(n)$, we define the spaces generated by $\mathbf{s}$ as $X_{k}(\mathbf{s})=\langle (k\cdot\chi^n_{s_i})_{i=0}^{d-1} \rangle_k$
and  $X_{\pm k}(\mathbf{s})=\langle (k\cdot\chi^n_{s_i})_{i=0}^{d-1} \rangle_{\pm k}$.
By Theorem \ref{Mil_Tay_fin} we obtain a canonicalization of a given coloring with respect to the types.

\begin{lem}
  \label{canon_color1}
  Let $m,n,k,r$ be positive integers, with $n\meg \mathrm{MT}(m,2m-1,r^\alpha)$, where $\alpha=\sum_{d=1}^m{d(k-1)^{d-1}}$. Then for every coloring of the set $X_k(n)$ with $r$ colors, there
  exists  $\mathbf{s}$ in $\mathrm{Block}^m(n)$ such that every $f,f'$ in $X_k(\mathbf{s})$ of the same type have the same color.
\end{lem}
\begin{proof}
  First let us observe that for every positive integer $d$ the number of types over $k$ of length $d$ is at most $d(k-1)^{d-1}$. Thus the number of types over $k$ of length at most $m$ is at most $\alpha$.
  Let $c:X_k(n)\to\{1,...,r\}$ be a coloring.
  Also let $\mathcal{T}$ be the set of all types over $k$ of length at most $m$ and $\mathcal{X}$ the set of all the maps from $\mathcal{T}$ into $\{1,...,r\}$. As we have already pointed out the set $\mathcal{T}$ has cardinality at most $\alpha$ and therefore the set $\mathcal{X}$ has cardinality at most $r^\alpha$. We define a coloring
  $\tilde{c}:\mathrm{Block}^m(n)\to\mathcal{X}$ as follows. For every $\mathbf{t}$ in $\mathrm{Block}^m(n)$, we first define $q_\mathbf{t}$ in $\mathcal{X}$ as follows. For every type $\varphi$ in $\mathcal{T}$
  we set $q_{\mathbf{t}}(\varphi)=c(\mathrm{map}(\varphi,\mathbf{t}|d))$, where $d$ is the length of $\varphi$. Finally, we set $\tilde{c}(\mathbf{t})=q_\mathbf{t}$ for all $\mathbf{t}$ in $\mathrm{Block}^m(n)$.

  Since $n\meg \mathrm{MT}(m,2m-1,r^\alpha)$, applying Theorem \ref{Mil_Tay_fin},
  we obtain a block sequence $\mathbf{s}'\in\mathrm{Block}^{2m-1}(n)$ such that the set $\mathrm{Block}^m(\mathbf{s'})$ is $\tilde{c}$-monochromatic.
  That is, there exists $q$ in $\mathcal{X}$ such that for every $\mathbf{t}$ in $\mathrm{Block}^m(\mathbf{s}')$ we have that
  $q_\mathbf{t}=q$.
  We set $\mathbf{s}=\mathbf{s}'|m$ and
  we observe that $\mathbf{s}$ is as desired. Indeed let $f,f'$ in $X_k(\mathbf{s})$ of the same type
  $\varphi$. Let $d$ be the length of $\varphi$. Clearly $1\mik d\mik m$. Since $\mathbf{s}$
  is the initial segment of $\mathbf{s}'$ of length $m$ and $\mathbf{s}'$ is of length $2m-1$, we can
  end-extend both $\mathrm{bsupp}(f)$ and $\mathrm{bsupp}(f')$ into $\mathbf{t}$ and
  $\mathbf{t}'$ respectively elements of $\mathrm{Block}^m(\mathbf{s}')$. Clearly, $f=\mathrm{map}(\varphi,\mathbf{t}|d)$ and $f'=\mathrm{map}(\varphi,\mathbf{t'}|d)$. Thus
  $c(f)=q_{\mathbf{t}}(\varphi)=q(\varphi)=q_{\mathbf{t}'}(\varphi)=c(f')$. The proof is complete.
\end{proof}
The above lemma is the main tool to obtain the color insensitivity by passing to a subspace.
In particular, we have the following.
\begin{lem}
  \label{insensitivity1}
  Let $m,n,k,r$ be positive integers, with \[n\meg \mathrm{MT}(m(2k-1),2m(2k-1)-1,r^\alpha),\] where $\alpha=\sum_{d=1}^{m(2k-1)}{d(k-1)^{d-1}}$. Then for every coloring $c$ of the set $X_k(n)$ with $r$ colors, there
  exists  $\mathbf{F}$ in $\mathrm{Block}^m_k(n)$ such that $c$ is insensitive over $\mathbf{F}$.
\end{lem}
\begin{proof}
  Let $c$ be a coloring of the set $X_k(n)$ with $r$ colors and $p=m(2k-1)$.
  Since $n\meg\mathrm{MT}(m(2k-1),2m(2k-1)-1,r^\alpha)$, applying Lemma \ref{canon_color1},
  we obtain a block sequence $\mathbf{s}=(s_i)_{i=0}^{p-1}\in\mathrm{Block}^p(n)$
  such that every $f,f'$ in $X_k(\mathbf{s})$ of the same type have the same color.
  The desired block sequence $\mathbf{F}=(f_i)_{i=0}^{m-1}$ is defined as follows.
  For every $i=0,...,m-1$, we set
  \[f_i=\sum_{q=-(k-1)}^{k-1}(k-|q|)\cdot\chi^n_{s_{j_i+q}},\]
  where $j_i=i(2k-1)+k-1$. It is easy to see that $\mathbf{F}$ belongs to $\mathrm{Block}^m_k(n)$
  and $\langle\mathbf{F}\rangle_k$ is a subset of $X_k(\mathbf{s})$. Thus every $f,f'$ in
  $\langle\mathbf{F}\rangle_k$ of the same type have the same color.
  Moreover, for every $i\in\{0,...,m-1\}$ the function $f_i$ has a ``pyramid'' shape and, in particular, for every $\ee\in\{0,...,k-1\}$ we have that
  \[T^{(\ee)}(f_i)(\max\supp (T^{(\ee)}(f_i)))=T^{(\ee)}(f_i)(\min\supp (T^{(\ee)}(f_i)))=1.\] This easily yields that for every $f,f'$ disjointly
  supported in $\langle\mathbf{F}\rangle_k$ the functions $f$ and $f+T^{(k-1)}(f')$ are of
  the same type, since $T^{(k-1)}(f')$ is of the form $\chi^n_s$ for some $s$ in $\mathrm{NU}(\mathbf{s})$
  disjoint to the support of $f$,
  and therefore of the same color. That is, the color $c$ is insensitive over $\mathbf{F}$.
\end{proof}

\section{Proof of Theorem \ref{positive_Gowers_fin}}
For the proof of Theorem \ref{positive_Gowers_fin} we need the following notation.
For $m, n$ positive integers with $m\mik n$ and $\mathbf{F}=(f_i)_{i=0}^{m-1}$ block sequence  in $X_{k}(n)$, we set
\[\langle \mathbf{F} \rangle_{[k]}=\Big\{\sum_{i=1}^{\ell}T^{\ee_i}(f_{j_i}):\ell\mik m,\; 0\mik j_1<...<j_\ell<m,\; \ee_1,...,\ee_\ell\in\{0,...,k-1\}\Big\}.
\]
\begin{proof}
  [Proof of Theorem \ref{positive_Gowers_fin}]
  We proceed by induction on $k$. The case $k=1$ of the theorem follows
  by Theorem \ref{Mil_Tay_fin} for ``$d=1$''. In particular, it is easy to observe that
  \begin{equation}
    \label{eq01}
    \mathrm{G}(1,m,r)=\mathrm{MT}(1,m,r)
  \end{equation}
  for every choice of positive integers $m$ and $r$.
  Assume that for some integer $k\meg2$, the theorem holds for $k-1$.
  We will establish the validity of the theorem for $k$, by showing that
  \begin{equation}
    \label{eq02}
    \mathrm{G}(k,m,r)\mik \mathrm{MT}(\mathrm{G}(k-1,m,r)(2k-1),2\mathrm{G}(k-1,m,r)(2k-1)-1,r^\alpha),
  \end{equation}
  where $\alpha=\sum_{d=1}^{\mathrm{G}(k-1,m,r)\cdot(2k-1)}{d(k-1)^{d-1}}$, for every choice of positive integers $m$ and $r$. Indeed, let $m,r$ be positive integers and
  set
  $M=\mathrm{G}(k-1,m,r)$. Also let $n$ be a positive integer with
  \begin{equation}
    \label{eq03}
    n\meg \mathrm{MT}(M(2k-1),2M(2k-1)-1,r^\alpha),
  \end{equation}
  where $\alpha=\sum_{d=1}^{M(2k-1)}{d(k-1)^{d-1}}$,
  and $c$ a coloring of the set $X_k(n)$ with $r$ colors.  By \eqref{eq03} and Lemma \ref{insensitivity1} applied
  for ``$m=M$'', we obtain $\mathbf{F'}=(f'_i)_{i=0}^{M-1}$ in $\mathrm{Block}^M_k(n)$ such that
  the coloring $c$ is insensitive over $\mathbf{F'}$.
  We define a map $Q:X_{[k-1]}(M)\to \langle\mathbf{F'}\rangle_{[k]}$ by setting
    \[Q(g)=\sum_{i\in\supp (g)}T^{(k-1-g(i))}(f'_i)\]
    for all $g\in X_{[k-1]}(M)$.
  Let us isolate the following easy to observe properties of the map $Q$.
  \begin{enumerate}
    \item[(a)] If $g\in X_{k-1}(M)$, then $Q(g)\in \langle\mathbf{F'}\rangle_k$.
    \item[(b)] If $(g_i)_{i=0}^{d-1}$ is a block sequence in $X_{k-1}(M)$ then $(Q(g_i))_{i=0}^{d-1}$ is also a block sequence in $\langle\mathbf{F'}\rangle_k$.
    \item[(c)] For every $g\in X_{[k-1]}(M)$, we have  $T(Q(g))=Q(T(g))+\sum_{i\in s}T^{(k-1)}(f'_i)$, where $s=\supp(g)\setminus\supp(T(g))$.
  \end{enumerate}
  Hence, if $(g_i)_{i=0}^{d-1}$ is a block sequence in $X_{k-1}(M)$ and $\ee_0,...,\ee_{d-1} \in
  \{0,...,k-2\}$ with $\min_{0\mik i<d}\ee_i=0$, then
  \[\sum_{i=0}^{d-1}T^{(\ee_i)}\big(Q(g_i)\big)=Q\Big(\sum_{i=0}^{d-1}T^{(\ee_i)}(g_i)\Big)+\sum_{i\in s}T^{(k-1)}(f'_i),\]
  where $s=\bigcup_{i=0}^{d-1}\supp(T^{(\max\{0,\ee_i-1\})}(g_i))\setminus\supp(T^{(\ee_i)}(g_i))$, and therefore, by the insensitivity of the coloring $c$ over $\mathbf{F'}$, we have that
  \begin{equation}
    \label{eq04}
    c\Big(Q\Big(\sum_{i=0}^{d-1}T^{(\ee_i)}g_i\Big)\Big)=c\Big(\sum_{i=0}^{d-1}T^{(\ee_i)}(Q(g_i))\Big).
  \end{equation}
  We define a coloring $\tilde{c}$ of the set $X_{k-1}(M)$ by setting $\tilde{c}(g)=c(Q(g))$ for all $g\in X_{k-1}(M)$. Let us point out that the coloring $\tilde{c}$ is well defined due to property (a) above.
  By the choice of $M$ and the inductive assumption, we obtain a block sequence $\mathbf{G}=(g_i)_{i=0}^{m-1}$
  in $X_{k-1}(M)$ such that the set $\langle\mathbf{G}\rangle_{k-1}$ is monochromatic with respect to $\tilde{c}$. Therefore, setting $\mathbf{F}=(f_i)_{i=0}^{m-1}=(Q(g_i))_{i=0}^{m-1}$, by the definition of $\tilde{c}$ and \eqref{eq04}, we have that the set
  \begin{equation}
    \label{eq05}
    \begin{split}\Big\{\sum_{i=1}^{\ell}T^{\ee_i}(f_{j_i}):\;&
\ell\mik m,\; 0\mik j_1<...<j_\ell<m,\; \\
&\ee_1,...,\ee_\ell\in\{0,...,k-2\}\;\text{and}\;\min_{1\mik i\mik\ell}\ee_i=0\Big\}
\end{split}
  \end{equation}
  is monochromatic with respect to $c$.
 Since $c$ is insensitive over $\mathbf{F}'$ and $\mathbf{F}$ is a block subsequence of $\mathbf{F}'$, by Fact \ref{fact_ins}, we get that that $c$ is insensitive over $\mathbf{F}$. Hence, by \eqref{eq05},
 we get that $\langle\mathbf{F}\rangle_k$ is monochromatic with respect to $c$ and the proof of the
 inductive step is complete.

 Finally, by (ii) of Proposition \ref{prim_rec_properties}, \eqref{eq01}, \eqref{eq02} and the fact that the numbers $\mathrm{MT}(d,m,r)$ are upper bounded by a function belonging to the class $\mathcal{E}^6$ of Grzegorczyk's hierarchy we have that the numbers $\mathrm{G}(k,m,r)$ are upper bounded by a function belonging to the class $\mathcal{E}^7$.
\end{proof}

\section{Proof of Theorem \ref{full_Gowers_fin}}

First let us extend the notion of the support. Let $m,n,k$ be positive integers with $m\mik n$ and
$\mathbf{s}$ in $\mathrm{Block}^m(n)$. We define the support of a function $f\in X_{\pm k}(\mathbf{s})$
with respect to $\mathbf{s}$ as
\[\supp_\mathbf{s}(f)=\supp(g),\]
where $g$ is the unique function in $X_{\pm k}(m)$ such that $f=\mathrm{map}(g,\mathbf{s})$.
For $f,f'$ in $X_{\pm k}(\mathbf{s})$, we will say that the pair $(f,f')$ is of $\mathbf{s}$-displacement
at most one if
\[\begin{split}&\min\supp_\mathbf{s}(f)\mik\min\supp_\mathbf{s}(f')\mik\min\supp_\mathbf{s}(f)+1\text{ and}\\
&\max\supp_\mathbf{s}(f)\mik\max\supp_\mathbf{s}(f')\mik\max\supp_\mathbf{s}(f)+1\end{split}\]
Moreover, for $\mathbf{F}=(f_i)_{i=0}^{d-1}$ and $\mathbf{F}'=(f'_i)_{i=0}^{d-1}$ block sequences in $X_{\pm k}(\mathbf{s})$ of some length $d\mik m$, we will say that the pair $(\mathbf{F},\mathbf{F}')$
is of $\mathbf{s}$-displacement at most one if the pair $(f_i,f'_i)$ is of $\mathbf{s}$-displacement at most one for all $0\mik i<d$. Finally, we will say that a block sequence
$\mathbf{F}=(f_i)_{i=0}^{d-1}$ in $X_{\pm k}(\mathbf{s})$ of some length $d\mik m$ is $\mathbf{s}$-skipped
block if $\max\supp_\mathbf{s}(f_i)+1\mik\min\supp_\mathbf{s}(f_{i+1})$ for all $0\mik i<d-1$. Under this terminology, we have the following immediate fact.
\begin{fact}
  \label{fact2}
  Let $d,m,n,k$ be positive integers with $d\mik m\mik n$ and $\mathbf{s}\in \mathrm{Block}^m(n)$. Also let $\mathbf{F}=(f_i)_{i=0}^{d-1}$ be an $\mathbf{s}$-skipped block sequence in $X_{\pm k}(\mathbf{s})$ and $f'_0,...,f'_{d-1}$ elements of $X_{\pm k}(\mathbf{s})$ such that the pair $(f_i,f'_i)$ is of $\mathbf{s}$-displacement at most one for all $0\mik i<d$. Then the sequence $(f'_i)_{i=0}^{d-1}$ is a
  block sequence in $X_{\pm k}(\mathbf{s})$ and the pair $(\sum_{i=0}^{d-1}f_i,\sum_{i=0}^{d-1}f'_i)$
  is of $\mathbf{s}$-displacement at most one.

  Moreover, if in addition we have $\rho_\infty(f_i,f'_i)\mik1$ for all $i=0,...,d-1$, then we have that
  $\rho_\infty(\sum_{i=0}^{d-1}f_i,\sum_{i=0}^{d-1}f'_i)\mik1$.
\end{fact}

The following functions posses central role in the proof of Theorem \ref{full_Gowers_fin} and
they are inspired by the approach in \cite{K}.
Again let us fix a triple of positive integers $m,n,k$ with $m\mik n$ and
$\mathbf{s}=(s_i)_{i=0}^{m-1}$ in $\mathrm{Block}^m(n)$.
For every $\delta\in \{1,...,k\}$ and $k-1\mik\ell\mik m-k$,
we define the function
\[q(\delta,\ell,\mathbf{s})=\sum_{j=-(\delta-1)}^{\delta-1}(-1)^j(\delta-|j|)\chi^n_{s_{\ell+j}},\]
while for every non-positive integer $\delta$ by $q(\delta,\ell,\mathbf{s})$ we denote the constant zero function.
The basic properties of the functions $q(\delta,\ell,\mathbf{s})$ are summarized by the following easy to prove
lemma.
\begin{lem}
  \label{properties_q}
  Let $\ell,m,n,\delta,k$ be positive integers with $m\mik n$, $k-1\mik\ell\mik m-k$ and $1\mik\delta\mik k$. Also let $\mathbf{s}\in \mathrm{Block}^m(n)$. Then we have the following.
  \begin{enumerate}
    \item[(i)] $T(q(\delta,\ell,\mathbf{s}))=q(\delta-1,\ell,\mathbf{s})$.
    \item[(ii)] If $\ell<m-k$, then setting $f=-q(\delta,\ell,\mathbf{s})$ and
    $f'=q(\delta,\ell+1,\mathbf{s})$ we have  that
    \begin{enumerate}
      \item[(a)] the pair $(f,f')$ is of $\mathbf{s}$-displacement at most one and
      \item[(b)] $\rho_\infty(f,f')=1$.
    \end{enumerate}
  \end{enumerate}
\end{lem}

Let us also extend the notion of the positive subspace for block sequences in $X_{\pm k}(n)$.
For $m,n,k$ positive integers with $m\mik n$ and $\mathbf{F}=(f_i)_{i=0}^{m-1}$ in $\mathrm{Block}_{\pm k}^m(n)$ we set
\[\begin{split}\langle \mathbf{F} \rangle_{ k}=\Big\{\sum_{i=1}^{\ell} T^{\ee_i}(f_{j_i}):\;&
\ell\mik m,\; 0\mik j_1<...<j_\ell<m,\; \\
&\ee_1,...,\ee_\ell\in\{0,...,k-1\}\;\text{and}\;\min_{1\mik i\mik\ell}\ee_i=0\Big\}.
\end{split}\]
The proof of Theorem \ref{full_Gowers_fin} makes use of the following lemma, which gathers the
properties of the functions $q(\delta,\ell,\mathbf{s})$ that we shall need.
\begin{lem}
  \label{properties_q2}
  Let $m,n,k$ be positive integers with $2km\mik n$ and $\mathbf{s}\in \mathrm{Block}^{2km}(n)$.
  We set $\mathbf{G}=(g_i)_{i=0}^{m-1}=(q(k,\ell_i,\mathbf{s}))_{i=0}^{m-1}$ where $\ell_i=2ki+k-1$
  for all $0\mik i<m$.
  Let $\mathbf{F}$ be a block sequence in $\langle\mathbf{G}\rangle_k$.
  Then for every $f\in\langle\mathbf{F}\rangle_{\pm k}$ there exist
  $f'\in X_{\pm k}(\mathbf{s})$ and $f''\in\langle\mathbf{F}\rangle_{k}$
  such that
  \begin{enumerate}
    \item[(i)] the functions $f',f''$ are of the same type,
    \item[(ii)] $\rho_\infty(f,f')\mik1$,
    \item[(iii)] the pair $(f,f')$ is of $\mathbf{s}$-displacement at most one and
    \item[(iv)] $\supp(f)=\supp(f'')$.
  \end{enumerate}
\end{lem}
\begin{proof}
  Let $d$ be the length of $\mathbf{F}$ and $\mathbf{F}=(f_j)_{j=0}^{d-1}$.
  Then for every $j=0,...,d-1$ there exist a subset $s_j$ of $\{0,..,m-1\}$ and a family $(\ee_i)_{i\in s_j}$ of elements from the set $\{0,...,k-1\}$ such that
  \begin{equation}
    \label{eq06}
    \min_{i\in s_j}\ee_i=0
  \end{equation}
  and $f_j=\sum_{i\in s_j}T^{(\ee_i)}(g_i)$.
  By the definition of $\mathbf{G}$ and part (i) of Lemma \ref{properties_q}, we get that
  \begin{equation}
    \label{eq07}
    f_j=\sum_{i\in s_j}q(k-\ee_i,\ell_i,\mathbf{s})
  \end{equation}
  for all $j=0,...,d-1$.

  Let $f\in\langle\mathbf{F}\rangle_{\pm k}$. Then there exist a subset $\tilde{s}$ of $\{0,...,d-1\}$,
  a family $(a_j)_{j\in \tilde{s}}$ of elements from $\{-1,1\}$ and a family $(\tilde{\ee}_j)_{j\in \tilde{s}}$
  of elements from $\{0,...,k-1\}$ such that
  \begin{equation}
    \label{eq08}
    \min_{j\in \tilde{s}}\tilde{\ee}_j=0
  \end{equation}
  and
  \begin{equation}
    \label{eq09}
    f=\sum_{j\in\tilde{s}}a_jT^{(\tilde{\ee}_j)}(f_j).
  \end{equation}
  Since $\mathbf{F}$ is a block subsequence of $\mathbf{G}$, we have that $(s_j)_{j=0}^{d-1}$ is
  a block sequence of nonempty finite subsets of $\nn$ and, in particular, consists of pairwise disjoint sets.
  Thus, setting $p$ to be the cardinality of the set $\cup_{j\in \tilde{s}}s_j$ and $\cup_{j\in \tilde{s}}s_j=\{i_0<...<i_{p-1}\}$, we have that for every $x=0,...,p-1$ there exists unique $j_x$ in $\tilde{s}$ such that $i_x\in s_{j_x}$. By equation \eqref{eq09} we have that
  \begin{equation}
    \label{eq10}
    \begin{split}
      f&=\sum_{j\in\tilde{s}}a_jT^{(\tilde{\ee}_j)}(f_j)
      \stackrel{\eqref{eq07}}{=}\sum_{j\in\tilde{s}}a_jT^{(\tilde{\ee}_j)}\Big(\sum_{i\in s_j}q(k-\ee_i,\ell_i,\mathbf{s})\Big) \\
      &=\sum_{j\in\tilde{s}}a_j\sum_{i\in s_j}T^{(\tilde{\ee}_j)}\big(q(k-\ee_i,\ell_i,\mathbf{s})\big)
      \;\;\;\;\;\;\;\;\;\;\; (\text{Part (i) of Lemma \ref{properties_q}})\\
      &=\sum_{j\in\tilde{s}}\sum_{i\in s_j}a_jq(k-\ee_i-\tilde{\ee}_j,\ell_i,\mathbf{s})
      =\sum_{x=0}^{p-1}a_{j_x}q(k-\ee_{i_x}-\tilde{\ee}_{j_x},\ell_{i_x},\mathbf{s}).
    \end{split}
  \end{equation}
  For every $x=0,...,p-1$ we set $\ell'_{i_x}=\ell_{i_x}$ if $a_{j_x}=1$ and $\ell'_{j_x}=\ell_{j_x}+1$ if $a_{j_x}=-1$. The sequence $(q(k,\ell_i,\mathbf{s}))_{i=0}^{m-1}$, by its definition, is $\mathbf{s}$-skipped block and therefore the sequence $(q(k-\ee_{i_x}-\tilde{\ee}_{j_x},\ell_{i_x},\mathbf{s}))_{x=0}^{p-1}$ is $\mathbf{s}$-skipped block.
  By part (ii) of Lemma \ref{properties_q}, for every $x=0,...,p-1$  we have that the pair
  \[(q(k-\ee_{i_x}-\tilde{\ee}_{j_x},\ell_{i_x},\mathbf{s}), q(k-\ee_{i_x}-\tilde{\ee}_{j_x},\ell'_{i_x},\mathbf{s}))\]
  is of $\mathbf{s}$-displacement at most one and
  \[\rho_\infty(q(k-\ee_{i_x}-\tilde{\ee}_{j_x},\ell_{i_x},\mathbf{s}), q(k-\ee_{i_x}-\tilde{\ee}_{j_x},\ell'_{i_x},\mathbf{s}))\mik1.\]
  Thus, by Fact \ref{fact2}, we have that $(q(k-\ee_{i_x}-\tilde{\ee}_{j_x},\ell'_{i_x},\mathbf{s}))_{x=0}^{p-1}$ is a block sequence and setting
  \[f'=\sum_{x=0}^{p-1}q(k-\ee_{i_x}-\tilde{\ee}_{j_x},\ell'_{i_x},\mathbf{s})\]
  we have that $\rho_\infty(f,f')\mik1$ and the pair $(f,f')$
  is of $\mathbf{s}$-displacement at most one, i.e., parts (ii) and (iii) of the lemma are satisfied. Moreover,
  again by the fact that the sequence $(q(k-\ee_{i_x}-\tilde{\ee}_{j_x},\ell'_{i_x},\mathbf{s}))_{x=0}^{p-1}$ is block and by part (i) of Lemma \ref{properties_q}, we have that
  \[f'=\sum_{x=0}^{p-1}q(k-\ee_{i_x}-\tilde{\ee}_{j_x},\ell'_{i_x},\mathbf{s})
    =\sum_{x=0}^{p-1}T^{(\ee_{i_x}+\tilde{\ee}_{j_x})}q(k,\ell'_{i_x},\mathbf{s}).
  \]
  Hence, by equations \eqref{eq06} and \eqref{eq08}, we have that $f'\in X_{\pm k}(\mathbf{s})$.

  We set
  \[f''=\sum_{x=0}^{p-1}q(k-\ee_{i_x}-\tilde{\ee}_{j_x},\ell_{i_x},\mathbf{s}).\]
  By \eqref{eq10} and the definition of $f''$, we have that part (iv) of the lemma is satisfied.
  Recalling that the sequence $(q(k-\ee_{i_x}-\tilde{\ee}_{j_x},\ell'_{i_x},\mathbf{s}))_{x=0}^{p-1}$
  is block, we observe that the functions $f'$ and $f''$ are of the same type and therefore part (i) of the
  lemma is satisfied.
  Finally, we have that
  \begin{equation}
  \label{eq11}
  \begin{split}
    f''&=\sum_{x=0}^{p-1}q(k-\ee_{i_x}-\tilde{\ee}_{j_x},\ell_{i_x},\mathbf{s})
    =\sum_{j\in\tilde{s}}\sum_{i\in s_j}q(k-\ee_i-\tilde{\ee}_j,\ell_i,\mathbf{s})\\
    &=\sum_{j\in\tilde{s}}\sum_{i\in s_j}T^{(\tilde{\ee}_j)}\big(q(k-\ee_i,\ell_i,\mathbf{s})\big)
    =\sum_{j\in\tilde{s}}T^{(\tilde{\ee}_j)}\Big(\sum_{i\in s_j}q(k-\ee_i,\ell_i,\mathbf{s})\Big)\\
    &\stackrel{\eqref{eq07}}{=}\sum_{j\in\tilde{s}}T^{(\tilde{\ee}_j)}(f_j),
  \end{split}
  \end{equation}
  where the third equality holds by part (i) of Lemma \ref{properties_q}.
  By \eqref{eq08} and \eqref{eq11}, we have that $f''\in\langle\mathbf{F}\rangle_k$. The proof is complete.
\end{proof}

Finally, identical arguments to the ones used in the proof of Lemma \ref{canon_color1} yield the following.
\begin{lem}
  \label{canon_color2}
  Let $m,n,k,r$ be positive integers, with $n\meg \mathrm{MT}(m,2m-1,r^\beta)$, where $\beta=\sum_{d=1}^m2d(2k-1)^{d-1}$. Then for every coloring of the set $X_{\pm k}(n)$ with $r$ colors, there
  exists  $\mathbf{s}$ in $\mathrm{Block}^m(n)$ such that every $f,f'$ in $X_{\pm k}(\mathbf{s})$ of the same type have the same color.
\end{lem}

Actually, we prove the following slightly stronger version of Theorem \ref{full_Gowers_fin}. For its proof, we will need some notation. For $m,n,k$ positive integers with $m\mik n$ and $\mathbf{F}=(f_i)_{i=0}^{m-1}$ in $\mathrm{Block}_{\pm k}^m(n)$ we set
\[\begin{split}\langle \mathbf{F} \rangle_{[k]}=\Big\{\sum_{i=1}^{\ell} T^{\ee_i}(f_{j_i}):
\ell\mik m,\; &0\mik j_1<...<j_\ell<m,\; \\
&\text{and}\;\ee_1,...,\ee_\ell\in\{0,...,k-1\}\Big\}.
\end{split}\]
\begin{thm}
  \label{strong_full_gowers}
  Let $m,n,k,r$ be positive integers with $n\meg\mathrm{MT}(2kM,4kM-1,r^\beta)$, where $M=\mathrm{G}(k,m,r)$ and $\beta=\sum_{d=1}^{2kM}2d(2k-1)^{d-1}$. Then for every coloring $c:X_{\pm k}(n)\to\{1,...,r\}$ there exist $\mathbf{s}\in\mathrm{Block}^{2kM}(n)$ and
  $\mathbf{F}\in\mathrm{Block}_{\pm k}^m(\mathbf{s})$ such that
  \begin{enumerate}
    \item[(i)] $\mathbf{F}$ is $\mathbf{s}$-skipped block and
    \item[(ii)] there exists $i_0$ in $\{1,...,r\}$ such that for every $f\in\langle\mathbf{F}\rangle_{\pm k}$ there exists $f'\in X_{\pm k}(\mathbf{s})$
        such that $c(f')=i_0$, $\rho_\infty(f,f')\mik1$ and the pair $(f,f')$ is of $\mathbf{s}$-displacement at most one.
  \end{enumerate}
        In particular, we have that $\langle\mathbf{F}\rangle_{\pm k}$ is approximately monochromatic.
\end{thm}

\begin{proof}
  Let $c:X_{\pm k}(n)\to\{1,...,r\}$ be a coloring. Since $n\meg\mathrm{MT}(2kM,4kM-1,r^\beta)$, by Lemma \ref{canon_color2} applied for ``$m=2kM$'',
  we have that there
  exists an $\mathbf{s}$ in $\mathrm{Block}^{2kM}(n)$ such that every $f,f'$ in $X_{\pm k}(\mathbf{s})$ of the same type have the same color.
  We set $\mathbf{G}=(g_i)_{i=0}^{M-1}=(q(k,\ell_i,\mathbf{s}))_{i=0}^{M-1}$ where $\ell_i=2ki+k-1$
  for all $0\mik i<M$ as in Lemma \ref{properties_q2}. Observe that $\mathbf{G}$ is $\mathbf{s}$-skipped block.

  We define a map $Q:X_{[k]}(M)\to\langle\mathbf{G}\rangle_{[k]}$ by setting
  \[Q(g)=\sum_{i=0}^{M-1}T^{(k-h(i))}(g_i)\]
  for all $g$ in $X_{[k]}(M)$. It is easy to observe that $Q$ is 1-1 and onto, while $Q$ restricted to
  $X_k(M)$ is 1-1 and onto $\langle\mathbf{G}\rangle_k$. Moreover, the image of
  every block sequence in $X_{k}(M)$ is a block sequence in $\langle\mathbf{G}\rangle_{k}$ and the pre-image of every block sequence in $\langle\mathbf{G}\rangle_{k}$
  is a block sequence in $X_{k}(M)$. Finally,
  for every $g$ in $X_{[k]}(M)$ we have that $Q(T(g))=T(Q(g))$. Under theses remarks, for every block sequence $\mathbf{H}=(h_i)_{i=0}^{d-1}$ in $X_{k}(M)$, setting $\mathbf{F}=(Q(h_i))_{i=0}^{d-1}$, we have that
  \begin{equation}
    \label{eq12}
    Q[\langle\mathbf{H}\rangle_k]=\langle\mathbf{F}\rangle_k.
  \end{equation}
  We define a coloring $\tilde{c}:X_{k}(M)\to\{1,...,r\}$ by
  setting $\tilde{c}(g)=c(Q(g))$ for all $g\in X_{k}(M)$. By the definition of $M$ and applying Theorem \ref{positive_Gowers_fin}
  we obtain a block sequence $\mathbf{H}=(h_i)_{i=0}^{m-1}$ in $X_{k}(M)$ of length $m$ such that the set $\langle\mathbf{H}\rangle_k$ is $\tilde{c}$-monochromatic.

  We set $\mathbf{F}=(Q(h_i))_{i=0}^{n-1}$. Then $\mathbf{F}$ is a block subsequence of $\mathbf{G}$ of length $m$. Moreover, since $\mathbf{G}$ is $\mathbf{s}$-skipped block, we
  have that $\mathbf{F}$ is also $\mathbf{s}$-skipped block, that is, part (i) of the theorem is
  satisfied.
  By the definition of $\tilde{c}$, the fact that $\langle\mathbf{H}\rangle_k$ is $\tilde{c}$-monochromatic and
  \eqref{eq12} we have that the set $\langle\mathbf{F}\rangle_k$ is $c$-monochromatic. Let $i_0$ in $\{1,...,r\}$ such that $c(f)=i_0$ for all $f$ in $\langle\mathbf{F}\rangle_k$.

  In order to check the validity of part (ii) of the theorem, we fix $f\in\langle\mathbf{F}\rangle_{\pm k}$.
  By the choice of $\mathbf{G}$, the fact that $\mathbf{F}$ is a block subsequence of $\mathbf{G}$ and
  Lemma \ref{properties_q2}, we have that there exist
  $f'\in X_{\pm k}(\mathbf{s})$ and $f''\in\langle\mathbf{F}\rangle_{k}$
  such that
  \begin{enumerate}
    \item[(i)] the functions $f',f''$ are of the same type,
    \item[(ii)] $\rho_\infty(f,f')\mik1$ and
    \item[(iii)] the pair $(f,f')$ is of $\mathbf{s}$-displacement at most one.
  \end{enumerate}
By (i) and the choice of $\mathbf{s}$, we have  $c(f')=c(f'')=i_0$. The proof is complete.
\end{proof}
Clearly, Theorem \ref{strong_full_gowers} yields Theorem \ref{full_Gowers_fin} and, in particular, for
every choice of positive integers $m,k,r$ we have that
\[\mathrm{G}_\pm(k,m,r)\mik\mathrm{MT}(2kM,4kM-1,r^\beta),\]
where $M=\mathrm{G}(k,m,r)$ and $\beta=\sum_{d=1}^{2kM}2d(2k-1)^{d-1}$. Therefore, since the numbers $\mathrm{MT}(d,m,r)$ are upper bounded by a function belonging to the class $\mathcal{E}^6$ of Grzegorczyk's hierarchy and the numbers $\mathrm{G}(k,m,r)$ are upper bounded by a function belonging to the class $\mathcal{E}^7$, by (ii) of Proposition \ref{prim_rec_properties}, we have that the numbers $\mathrm{G}_\pm(k,m,r)$ are upper bounded by a function belonging to the class $\mathcal{E}^7$.

\section{Multidimensional Versions}\label{section_mult_versions}
In this section we provide multidimensional versions of the Theorems \ref{positive_Gowers_fin} and \ref{full_Gowers_fin}. A standard iteration of Theorem \ref{positive_Gowers_fin} yields the following.

\begin{thm}
  \label{mult_positive_Gowers_fin}
  For every quadrat of positive integers $k,d,m,r$ with $d\mik m$, there exists a positive integer $n_0$ satisfying the
  following property. For every integer $n\meg n_0$ and every coloring of the set $\mathrm{Block}_k^d(n)$ with $r$ colors, there exists a block sequence $\mathbf{F}$ in $X_k(n)$ of length $m$ such that the set $\mathrm{Block}_k^d( \mathbf{F} )$ is monochromatic.
  We denote the least $n_0$ satisfying the above property by $\mathrm{MG}(k,d,m,r)$.

  Moreover, the numbers $\mathrm{MG}(k,d,m,r)$ are upper bounded by a primitive recursive function belonging to the class $\mathcal{E}^9$ of Grzegorczyk's hierarchy.
\end{thm}

The metric $\rho_{\infty}$ defined on $X_{\pm k}(n)$ naturally induces a metric on $\mathrm{Block}_{\pm k}^d(n)$, which we denote, abusing notation, by $\rho_{\infty}$. In particular, for every
$\mathbf{F}=(f_i)_{i=0}^{d-1}$ and $\mathbf{G}=(g_i)_{i=0}^{d-1}$ in $\mathrm{Block}_{\pm k}^d(n)$, we define the distance between $\mathbf{F}$ and $\mathbf{G}$ as
\[\rho_\infty(\mathbf{F},\mathbf{G})=\max_{0\mik i<d}\rho_{\infty}(f_i,g_i).\]
Finally, given a finite coloring  $c:\mathrm{Block}_{\pm k}^d(n)\to\{1,...,r\}$, we say that a subset
$A$ of $\mathrm{Block}_{\pm k}^d(n)$ is \textit{approximately monochromatic} if there exists some $i_0\in\{1,...,r\}$ such
that for every $\mathbf{F}$ in $A$ there exists  $\mathbf{F}'$ in $\mathrm{Block}_{\pm k}^d(n)$ with $c(\mathbf{F}')=i_0$ and $\rho_\infty(\mathbf{F},\mathbf{F}')\mik1$. We have the following multidimensional version
of Theorem \ref{full_Gowers_fin}.

\begin{thm}
  \label{mult_ful_Gowers_fin}
  For every quadrat of positive integers $k,d,m,r$ with $d\mik m$, there exists a positive integer $n_0$ satisfying the
  following property. For every integer $n\meg n_0$ and every coloring of the set $\mathrm{Block}_{\pm k}^d(n)$ with $r$ colors, there exists a block sequence $\mathbf{F}$ in $X_{\pm k}(n)$ of length $m$ such that the set $\mathrm{Block}_{\pm k}^d( \mathbf{F} )$ is approximately monochromatic.  We denote the least $n_0$ satisfying the above property by $\mathrm{MG}_{\pm}(k,d,m,r)$.

  Moreover, the numbers $\mathrm{MG}_{\pm}(k,d,m,r)$ are upper bounded by a primitive recursive function belonging to the class $\mathcal{E}^9$ of Grzegorczyk's hierarchy.
\end{thm}

\subsection{Proof of Theorem \ref{mult_positive_Gowers_fin}} The method for obtaining Theorem \ref{mult_positive_Gowers_fin} from Theorem \ref{positive_Gowers_fin}
is quit standard. However, we include it for seasons of completeness.
For every two finite sequences $\mathbf{F}$ and $\mathbf{G}$ by $\mathbf{F}^\con\mathbf{G}$ be denote the
concatenation of $\mathbf{F}$ and $\mathbf{G}$.

\begin{lem}
  \label{stab_d}
  Let $d,\ell_1,N,N',n,k,r$ be positive integers such that  $d\mik\ell_1$, $N=\mathrm{G}(k,N',r^{(k+1)^{d\ell_1}})$ and $\ell_1+N\mik n$. Also let
  $\mathbf{G}$ and $\mathbf{F}$ be block sequences in $X_k(n)$ such that
  \begin{enumerate}
    \item[(ii)] $\mathbf{G}$ is of length at most $\ell_1$,
    \item[(ii)] $\mathbf{F}$ is of length $N$ and
    \item[(iii)] $\mathbf{G}^\con\mathbf{F}$ is a block sequence.
  \end{enumerate}
  Finally, let $c:\mathrm{Block}_k^{d+1}(n)\to\{1,...,r\}$ be a coloring. Then there exists
  $\mathbf{F}'$ in $\mathrm{Block}_k^{N'}(\mathbf{F})$ such that for every $H$ in
  $\mathrm{Block}_k^d(\mathbf{G})$ and every $f,f'$ in $\langle\mathbf{F}'\rangle_k$
  we have that $c(\mathbf{H}^\con(f))=c(\mathbf{H}^\con(f'))$.
\end{lem}
\begin{proof}
  Let $\mathcal{X}$ be the set of all functions from $\mathrm{Block}_k^d(\mathbf{G})$ into $\{1,...,r\}$.
  It is easy to observe that the set $\mathrm{Block}_k^d(\mathbf{G})$ is of cardinality at most $(k+1)^{d\ell_1}$ and therefore the set $\mathcal{X}$ is of cardinality at most $r^{(k+1)^{d\ell_1}}$.
  We define a coloring $\tilde{c}:\langle\mathbf{F}\rangle_k\to\mathcal{X}$ by setting
  \[\tilde{c}(f)(\mathbf{H})=c(\mathbf{H}^\con(f))\]
  for all $\mathbf{H}$ in $\mathrm{Block}_k^d(\mathbf{G})$ and $f\in\langle\mathbf{F}\rangle_k$.
  By the choice of $N$ and Theorem \ref{positive_Gowers_fin}, we have that there exists $\mathbf{F}'$
  in $\mathrm{Block}_k^{N'}(\mathbf{F})$ such that the set $\langle\mathbf{F}'\rangle_k$
  is $\tilde{c}$-monochromatic. Thus for every $f,f'$ in $\langle\mathbf{F}'\rangle_k$ we have that
  $\tilde{c}(f)=\tilde{c}(f')$ and therefore for all $\mathbf{H}$ in $\mathrm{Block}_k^d(\mathbf{G})$
  we have that $c(\mathbf{H}^\con(f))=\tilde{c}(f)(\mathbf{H})=\tilde{c}(f')(\mathbf{H})=c(\mathbf{H}^\con(f'))$.
  The proof is complete.
\end{proof}

The next lemma follows by an iterated use of the above lemma. We will need the following invariants.
In particular, we define a function $h:\nn^5\to\nn$ as follows. For every positive integers $d,\ell,r,k$
we inductively define
\begin{equation}
\label{eq13}
\left\{ \begin{array} {l} h(d,\ell,r,k,0)=0\\
                          h(d,\ell,r,k,x+1)=\mathrm{G}(k,h(d,\ell,r,k,x)+1,r^{(k+1)^{d\ell}}).\end{array}  \right.
\end{equation}
Since the numbers $\mathrm{G}(k,m,r)$ are upper bounded by a function belonging to the class $\mathcal{E}^7$ of Grzegorczyk's hierarchy, by (ii) of Proposition \ref{prim_rec_properties}, we have that the function $h$ is upper bounded by a function belonging to the class $\mathcal{E}^8$.

\begin{lem}
  \label{canon_d}
  Let $d,\ell,n,k,r$ be positive integers such that $d<\ell$ and $n\meg d+h(d,\ell,r,k,\ell-d)$.
  Also let $c:\mathrm{Block}_k^{d+1}(n)\to\{1,...,r\}$ be a coloring. Then there exists
  $\mathbf{G}$ in $\mathrm{Block}_k^\ell(n)$ such that for every $\mathbf{J},\mathbf{J}'$ in
  $\mathrm{Block}_k^d(\mathbf{G})$ with $\mathbf{J}|d=\mathbf{J}'|d$, we have that $c(\mathbf{J})=c(\mathbf{J}')$.
\end{lem}
\begin{proof}
  For every $p=0,...,\ell-d$ we set $N_p=h(d,\ell,r,k,\ell-d-p)$. Then by \eqref{eq13} we have that
  \begin{equation}
    \label{eq14}
    N_{p-1}=\mathrm{G}(k,N_p+1,r^{(k+1)^{d\ell}})
  \end{equation}
  for every $p$ in $\{1,...,\ell-d\}$. We also set $\mathbf{G}_0=(k\chi^n_{\{i\}})_{i=0}^{d-1}$
  and $\mathbf{F}_0=(k\chi^n_{\{i+d-1\}})_{i=0}^{N_0}$.
  We inductively construct two sequences $(\mathbf{F}_p)_{p=0}^{\ell-p}$
   and $(\mathbf{G}_p)_{p=0}^{\ell-p}$ satisfying the following for every $p=0,...,\ell-d$.
  \begin{enumerate}
    \item[(C1)] $\mathbf{G}_p$ belongs to $\mathrm{Block}_k^{d+p}(n)$ and $\mathbf{F}_p=(f^p_i)_{i=0}^{N_p}$ belongs to $\mathrm{Block}_k^{N_p+1}(n)$.
    \item[(C2)] If $p<\ell+d$, then $\mathbf{G}_{p}^\con\mathbf{F}_p^*$ is a block sequence,
    where $\mathbf{F}_{p}^*=(f^{p}_{i+1})_{i=0}^{N_{p}-1}$.
    \item[(C3)] If $p\meg1$, then $\mathbf{F}_p$ belongs to $\mathrm{Block}_k^{N_p}(\mathbf{F}^*_{p-1})$.
    \item[(C4)] If $p\meg1$, then $\mathbf{G}_p=\mathbf{G}_{p-1}^\con(f_0^p)$.
    \item[(C5)] If $p\meg1$, then for every $\mathbf{H}$ in $\mathrm{Block}_k^d(\mathbf{G}_{p-1})$ and $f,f'$ in
    $\langle\mathbf{F}_p\rangle_k$, we have that $c(\mathbf{H}^\con(f))=c(\mathbf{H}^\con(f'))$.
  \end{enumerate}
  Observe that for $p=0$ conditions C1 and C2 are satisfied, while C3-C5 are meaningless. The inductive step of the construction is a straightforward application of
  Lemma \ref{stab_d}. Indeed, assume that for some $p$ in $\{1,...,\ell-d\}$ the sequences
  $(\mathbf{G}_i)_{i=0}^{p-1}$ and $(\mathbf{F}_i)_{i=0}^{p-1}$ have been properly chosen.
  By \eqref{eq14}, applying Lemma \ref{stab_d} for ``$\ell_1=\ell\meg d+p-1$'', ``$N=N_{p-1}$'', ``$N'=N_p+1$'',
  ``$\mathbf{G}=\mathbf{G}_{p-1}$'' and ``$\mathbf{F}=\mathbf{F}_{p-1}^*$'', we obtain $\mathbf{F}_{p}\in \mathrm{Block}_k^{N_p+1}(F_{p-1}^*)$ such that for every $H$ in
  $\mathrm{Block}_k^d(\mathbf{G}_{p-1})$ and every $f,f'$ in $\langle\mathbf{F}_p\rangle_k$
  we have that $c(\mathbf{H}^\con(f))=c(\mathbf{H}^\con(f'))$. Clearly conditions C1, C3 and C5 are satisfied. Setting $\mathbf{G}_p=\mathbf{G}_{p-1}^\con(f_0^p)$ the proof of the inductive step is complete.
  We set $\mathbf{G}=\mathbf{G}_{\ell-d}$. It is easy to check that $\mathbf{G}$ is as desired.
\end{proof}

\begin{proof}[Proof of Theorem \ref{mult_positive_Gowers_fin}]
  We proceed by induction on $d$. The base case $d=1$ is established by Theorem \ref{positive_Gowers_fin}.
  In particular, we have
  \begin{equation}
  \label{eq15}
    \mathrm{MG}(k,1,m,r)=\mathrm{G}(k,m,r)
  \end{equation}
  for every choice of positive integers $k,m,r$.
  Assume that the theorem holds for some positive integer $d$ and let $m,k,r$ be positive integers.
  We will prove the inductive step by showing that
  \begin{equation}
    \label{eq16}
    \mathrm{MG}(k,d+1,m,r)\mik d+h(d,M+1,r,k,M-d+1),
  \end{equation}
  where $M=\mathrm{MG}(k,d,m-1,r)$ and $h$ is as defined in \eqref{eq13}.
  Indeed, let $n$ be a positive integer
  with
  \begin{equation}
    \label{eq17}
    n\meg d+h(d,M+1,r,k,M-d+1).
  \end{equation}
  Also let $c:\mathrm{Block}_k^{d+1}(n)\to\{1,...,r\}$ be a coloring. By \eqref{eq17} and Lemma \ref{canon_d}
  for ``$\ell=M+1$'', we have that there exists $\mathbf{G}=(g_i)_{i=0}^{M}$ in $\mathrm{Block}_k^{M+1}(n)$ such that for every $\mathbf{J},\mathbf{J}'$ in
  $\mathrm{Block}_k^d(\mathbf{G})$ with $\mathbf{J}|d=\mathbf{J}'|d$, we have that $c(\mathbf{J})=c(\mathbf{J}')$.
  We set $\mathbf{G}^*=(g_i)_{i=0}^{M-1}$. We define a coloring $\tilde{c}:\mathrm{Block}^d_k(\mathbf{G}^*)\to\{1,...,r\}$ by setting
  \[\tilde{c}(\mathbf{H})=c(\mathbf{H}^\con(g_M)).\]
  By the inductive assumption and the definition of $M$ we obtain $\mathbf{F}^*$ in $\mathrm{Block}_k^{m-1}(\mathbf{G}^*)$
  such that the set $\mathrm{Block}_k^d(\mathbf{F}^*)$ is $\tilde{c}$-monochromatic. We set $\mathbf{F}=\mathbf{F}^{*\con}(g_M)$. It is easy to see that $\mathbf{F}$ is as desired.

  Finally, since the numbers $\mathrm{G}(k,m,r)$ are upper bounded by a function belonging to the class $\mathcal{E}^7$ of Grzegorczyk's hierarchy and $h$ is upper bounded by a function belonging to the class $\mathcal{E}^8$, by \eqref{eq15}, \eqref{eq16} and (ii) of Proposition \ref{prim_rec_properties}, we have that the numbers $\mathrm{MG}(k,d,m,r)$ are upper bounded by a function belonging to $\mathcal{E}^9$.
\end{proof}

\subsection{Proof of Theorem \ref{mult_ful_Gowers_fin}}
The reduction of Theorem \ref{mult_ful_Gowers_fin} to Theorem \ref{mult_positive_Gowers_fin}
is similar to the one of Theorem \ref{full_Gowers_fin} to Theorem \ref{positive_Gowers_fin}.
We will need the analogues of Lemmas \ref{canon_color1} and \ref{properties_q2} for block sequences.
To this end, we define the type of a block sequence $\mathbf{F}=(f_i)_{i=0}^{m-1}$ in $X_{[\pm k]}(n)$ by setting
\[\mathrm{tp}(\mathbf{F})=(\mathrm{tp}(f_i))_{i=0}^{m-1}.\]
Moreover for every $k, d, m,n$ positive integers, with $d\mik m\mik n$, and $\mathbf{s}$ in $\mathrm{Block}^m(n)$, by
$\mathrm{Block}^d_{\pm k}(\mathbf{s})$ we denote the set of all block sequences in $X_{\pm k}(\mathbf{s})$
of length $d$.

\begin{lem}
  \label{canon_color_block}
  Let $d,m,n,k,r$ be positive integers, with $d\mik m$ and
  \[n\meg \mathrm{MT}(m,2m-1,r^\gamma),\]
  where $\gamma={m \choose d}2^dm^d(2k-1)^{d(m-1)}$. Then for every coloring of the set $\mathrm{Block}^d_{\pm k}(n)$ with $r$ colors, there
  exists  $\mathbf{s}$ in $\mathrm{Block}^m(n)$ such that every $\mathbf{F},\mathbf{F}'$ in $\mathrm{Block}^d_{\pm k}(\mathbf{s})$ of the same type have the same color.
\end{lem}

\begin{proof}
  We set $\mathcal{T}$ to be the set of sequences $\overline{\varphi}=(\varphi_i)_{i=0}^{d-1}$ of length $d$, such that  $\varphi_i$ is a type over $\pm k$ for all $0\mik i<d$ and $\sum_{i=0}^{d-1}|\varphi_i|\mik m$.
  Observing that the cardinality of the set of all types over $\pm k$ of some length
  $\ell\mik m$ is upper bounded by $2\ell(2k-1)^{\ell-1}\mik2m(2k-1)^{m-1}$, it easy to check
  that the cardinality of $\mathcal{T}$ is at most $\gamma$.
  Therefore, setting $\mathcal{X}$ to be the set of all maps from $\mathcal{T}$ into
  $\{1,...,r\}$, we have that the cardinality of $\mathcal{X}$ is at most $r^\gamma$.

  Let $c:\mathrm{Block}^d_{\pm k}(n)\to\{1,...,r\}$ be a coloring.
  Next we define a new coloring
  $\tilde{c}:\mathrm{Block}^m(n)\to\mathcal{X}$ as follows. For every $\overline{\varphi}=(\varphi_i)_{i=0}^{d-1}$ in $\mathcal{T}$ we define $\ell^{\overline{\varphi}}_0=0$ and $\ell_i^{\overline{\varphi}}=\sum_{j=0}^{i-1}|\varphi_j|$ for all $1\mik i<d$.
  Moreover, for every $\overline{\varphi}=(\varphi_i)_{i=0}^{d-1}$ in $\mathcal{T}$ and
  every $\mathbf{t}=(t_i)_{i=0}^{m-1}$ in $\mathrm{Block}^m(n)$  we define $\mathrm{bl}(\overline{\varphi},\mathbf{t})$
  in $\mathrm{Block}_{\pm k}^d(\mathbf{t})$ by the rule
  \[\mathrm{bl}(\overline{\varphi},\mathbf{t})=\big(\mathrm{map}(\varphi_i,(t_{\ell_i^{\overline{\varphi}}+j})_{j=0}^{|\varphi_i|-1})\big)_{i=0}^{d-1}.\]
  Clearly, for every choice of $\overline{\varphi}$ and $\mathbf{t}$ as above, we have that
  $\mathrm{bl}(\overline{\varphi},\mathbf{t})$ is of type $\overline{\varphi}$.
  Finally, for every $\mathbf{t}$ in $\mathrm{Block}^m(n)$ we define an element
  $q_\mathbf{t}$ of $\mathcal{X}$ by setting for every $\overline{\varphi}$ in $\mathcal{T}$
  \[q_\mathbf{t}(\overline{\varphi})=c(\mathrm{bl}(\overline{\varphi},\mathbf{t})).\]
  We define $\tilde{c}$ by setting $\tilde{c}(\mathbf{t})=q_\mathbf{t}$ for all $\mathbf{t}$ in
  $\mathrm{Block}^m(n)$.

  Since $n\meg \mathrm{MT}(m,2m-1,r^\gamma)$, applying Theorem \ref{Mil_Tay_fin},
  we obtain a block sequence $\mathbf{s}'\in\mathrm{Block}^{2m-1}(n)$ such that the set $\mathrm{Block}^m(\mathbf{s'})$ is $\tilde{c}$-monochromatic.
  That is, there exists $q$ in $\mathcal{X}$ such that for every $\mathbf{t}$ in $\mathrm{Block}^m(\mathbf{s}')$ we have that
  $\tilde{c}(\mathbf{t})=q_\mathbf{t}=q$.
  We set $\mathbf{s}=\mathbf{s}'|m$ and
  we observe that $\mathbf{s}$ is as desired. Indeed, let
  $\mathbf{F}=(f_i)_{i=0}^{d-1}$ and $\mathbf{F}'=(f'_i)_{i=0}^{d-1}$ in $\mathrm{Block}^d_{\pm k}(\mathbf{s})$ of the same type
  $\overline{\varphi}=(\varphi_i)_{i=0}^{d-1}$. Also let $\ell=\sum_{i=0}^{d-1}|\varphi_i|$.
  Clearly $d\mik\ell\mik m$. Also let
  \[\mathbf{t}_1=\mathrm{bsupp}(f_1)^\con...^\con\mathrm{bsupp}(f_{d-1})\;\text{and}\;
  \mathbf{t}'_1=\mathrm{bsupp}(f'_1)^\con...^\con\mathrm{bsupp}(f'_{d-1}).\]
  Clearly, both $\mathbf{t}_1$ and $\mathbf{t}'_1$ are of length $\ell$.
  Since $\mathbf{s}$
  is the initial segment of $\mathbf{s}'$ of length $m$ and $\mathbf{s}'$ is of length $2m-1$, we can
  end-extend both $\mathbf{t}_1$ and $\mathbf{t}'_1$ into $\mathbf{t}$ and
  $\mathbf{t}'$ respectively elements of $\mathrm{Block}^m(\mathbf{s}')$. Since $\mathbf{F}=\mathrm{bl}(\overline{\varphi},\mathbf{t})$ and $\mathbf{F}'=\mathrm{bl}(\overline{\varphi},\mathbf{t}')$, we get that
  $c(\mathbf{F})=q_{\mathbf{t}}(\overline{\varphi})=q(\overline{\varphi})=q_{\mathbf{t}'}(\overline{\varphi})=c(\mathbf{F}')$. The proof is complete.
\end{proof}

By Lemma \ref{properties_q2}, we have the following consequence.
To state it we need a slight modification of the existing terminology.
For $k,d,n$ positive integers with $d\mik n$ and $\mathbf{F}$ a block sequence in $X_{\pm k}(n)$ of length
at least $d$, we denote by $\mathrm{Block}_k^d(\mathbf{F})$ the set of all block sequences in $\langle\mathbf{F}\rangle_k$ of length $d$.

\begin{cor}
  \label{cor_prop_block_displ}
  Let $d,m,n,k$ be positive integers, with $d\mik m$ and $2km\mik n$, and $\mathbf{s}\in \mathrm{Block}^{2km}(n)$.
  We set $\mathbf{G}=(g_i)_{i=0}^{m-1}=(q(k,\ell_i,\mathbf{s}))_{i=0}^{m-1}$ where $\ell_i=2ki+k-1$
  for all $0\mik i<m$.
  Let $\mathbf{F}$ be a block sequence in $\langle\mathbf{G}\rangle_k$.
  Then for every $\mathbf{H}=(h_i)_{i=0}^{d-1}$ in $\mathrm{Block}_{\pm k}^d(\mathbf{F})$ there exist
  $\mathbf{H}'$ in $\mathrm{Block}_{\pm k}^d(\mathbf{s})$ and
  $\mathbf{H}''=(h''_i)_{i=0}^{d-1}$ in $\mathrm{Block}_k^d(\mathbf{F})$
  such that
  \begin{enumerate}
    \item[(i)] the sequences $\mathbf{H}',\mathbf{H}''$ are of the same type,
    \item[(ii)] $\rho_\infty(\mathbf{H},\mathbf{H}')\mik1$,
    \item[(iii)] the pair $(\mathbf{H},\mathbf{H}')$ is of $\mathbf{s}$-displacement at most one and
    \item[(iv)] $\supp(h_i)=\supp(h''_i)$, for all $0\mik i<d$.
  \end{enumerate}
\end{cor}
Actually, we prove the following slightly stronger version of Theorem \ref{mult_ful_Gowers_fin}.
\begin{thm}
  \label{mult_strong_full_gowers}
  Let $d,m,n,k,r$ be positive integers with $d\mik m$ and
  \begin{equation}
  \label{eq18}
  n\meg\mathrm{MT}(2kM,4kM-1,r^\gamma),
  \end{equation}
  where $M=\mathrm{MG}(k,d,m,r)$ and $\gamma={2kM \choose d}2^d(2kM)^d(2k-1)^{d(2kM-1)}$. Then for every coloring $c:\mathrm{Block}_{\pm k}^d(n)\to\{1,...,r\}$ there exist $\mathbf{s}\in\mathrm{Block}^{2kM}(n)$ and
  $\mathbf{F}\in\mathrm{Block}_{\pm k}^m(\mathbf{s})$ such that
  \begin{enumerate}
    \item[(i)] $\mathbf{F}$ is $\mathbf{s}$-skipped block and
    \item[(ii)] there exists $i_0$ in $\{1,...,r\}$ such that for every $\mathbf{H}\in\mathrm{Block}_{\pm k}^d(\mathbf{F})$ there exists $\mathbf{H}'\in\mathrm{Block}_{\pm k}^d(\mathbf{s})$
        such that $c(\mathbf{H}')=i_0$, $\rho_\infty(\mathbf{H},\mathbf{H}')\mik1$ and the pair $(\mathbf{H},\mathbf{H}')$ is of $\mathbf{s}$-displacement at most one.
  \end{enumerate}
        In particular, we have that $\mathrm{Block}_{\pm k}^d(\mathbf{F})$ is approximately monochromatic.
\end{thm}

\begin{proof}
  Let $c:\mathrm{Block}_{\pm k}(n)\to\{1,...,r\}$ be a coloring.
  By \eqref{eq18} and Lemma \ref{canon_color_block}, applied for ``$m=2kM$'',
  there
  exists an $\mathbf{s}$ in $\mathrm{Block}^{2kM}(n)$ such that every $\mathbf{H},\mathbf{H}'$ in $\mathrm{Block}^d_{\pm k}(\mathbf{s})$ of the same type have the same color.
  We define $\mathbf{G}=(g_i)_{i=0}^{M-1}=(q(k,\ell_i,\mathbf{s}))_{i=0}^{M-1}$, where $\ell_i=2ki+k-1$
  for all $0\mik i<M$, as in Corollary \ref{cor_prop_block_displ}.
  We define a map $Q:X_{[k]}(M)\to\langle\mathbf{G}\rangle_{[k]}$, as in the proof of Theorem
  \ref{full_Gowers_fin}, by setting
  \[Q(g)=\sum_{i=0}^{M-1}T^{(k-h(i))}(g_i)\]
  for all $g$ in $X_{[k]}(M)$. Moreover, we define a map
  $Q_d:\mathrm{Block}_{k}^d(M)\to\mathrm{Block}_{k}^d(\mathbf{G})$ setting
  \[Q_d(\mathbf{H}')=(Q(h'_i))_{i=0}^{d-1}\]
  for every $\mathbf{H}'=(h'_i)_{i=0}^{d-1}$ in $\mathrm{Block}_{k}^d(M)$.
  It is easy to observe that $Q_d$ is 1-1 and onto. Moreover, for every $\mathbf{F}'=(f'_i)_{i=0}^{\ell-1}$ block sequence in $X_k(M)$, setting $\mathbf{F}=(Q(f'_i))_{i=0}^{\ell-1}$, we have that the restriction of $Q_d$ on $\mathrm{Block}_{k}^d(\mathbf{F}')$ is 1-1 and onto $\mathrm{Block}_{k}^d(\mathbf{F})$.
  Finally, we define a coloring $\tilde{c}:\mathrm{Block}_{k}^d(M)\to\{1,...,r\}$, setting
  $\tilde{c}(\mathbf{H}')=c(Q(\mathbf{H}'))$ for all $\mathbf{H}'$ in $\mathrm{Block}_{k}^d(M)$.

  By the choice of $M$ and Theorem \ref{mult_positive_Gowers_fin},
  applied for the coloring $\tilde{c}$, we have that
  there exists a block sequence $\mathbf{F}'=(f'_i)_{i=0}^{m-1}$ in $X_k(n)$ of length $m$ such that the set $\mathrm{Block}_k^d( \mathbf{F}' )$ is monochromatic with respect to $\tilde{c}$.
  We set $\mathbf{F}=(Q(f'_i))_{i=0}^{m-1}$. By the definition of the coloring $\tilde{c}$,
  we have that $\mathrm{Block}_k^d(\mathbf{F})$ is monochromatic with respect to $c$.
  That is, there exists $j_0$ in $\{1,...,r\}$ such that $c(\mathbf{H})=j_0$ for all $\mathbf{H}$ in
  $\mathrm{Block}_k^d(\mathbf{F})$.

  Let us observe that $\mathbf{F}$ is as desired.
  Indeed, observe that $\mathbf{G}$ is an $\mathbf{s}$-skipped block sequence and therefore,
  since $\mathbf{F}$ is a block subsequence of $\mathbf{G}$, we have that $\mathbf{F}$ is an $\mathbf{s}$-skipped block sequence too, that is, part (i) of the theorem is satisfied.
  Let $\mathbf{H}$ in $\mathrm{Block}_{\pm k}^d(\mathbf{F})$. By the definition of
  $\mathbf{G}$, the fact that $\mathbf{F}$ is a block sequence in $\langle\mathbf{G}\rangle_k$
  and Corollary \ref{cor_prop_block_displ}, we have that
  there exist
  $\mathbf{H}'$ in $\mathrm{Block}_{\pm k}^d(\mathbf{s})$ and
  $\mathbf{H}''$ in $\mathrm{Block}_k^d(\mathbf{F})$
  such that
  \begin{enumerate}
    \item[(i)] the sequences $\mathbf{H}',\mathbf{H}''$ are of the same type,
    \item[(ii)] $\rho_\infty(\mathbf{H},\mathbf{H}')\mik1$ and
    \item[(iii)] the pair $(\mathbf{H},\mathbf{H}')$ is of $\mathbf{s}$-displacement at most one.
    \end{enumerate}
  Since $\mathbf{H}',\mathbf{H}''$ both belong to $\mathrm{Block}_{\pm k}^d(\mathbf{s})$, by the choice
  of $\mathbf{s}$ and by (i) above, we have that
  \begin{equation}
    \label{eq20}
    c(\mathbf{H}')=c(\mathbf{H}'').
  \end{equation}
  Since $\mathbf{H}''$ belongs to $\mathrm{Block}_k^d(\mathbf{F})$, we have that
  \begin{equation}
    \label{eq21}
    c(\mathbf{H}'')=j_0.
  \end{equation}
  By \eqref{eq20} and \eqref{eq21}, we have that $c(\mathbf{H}')=j_0$. By (ii) and (iii) above, the
  proof is complete.
\end{proof}

  Clearly, Theorem \ref{mult_strong_full_gowers} yields Theorem \ref{mult_ful_Gowers_fin} and in particular, for
every choice of positive integers $d,m,k,r$  with $d\mik m$ we have that
\[\mathrm{MG}_{\pm}(k,d,m,r)\mik  \mathrm{MT}(2kM,4kM-1,r^\gamma),
  \]
  where $M=\mathrm{MG}(k,d,m,r)$ and $\gamma={2kM \choose d}2^d(2kM)^d(2k-1)^{d(2kM-1)}$.
  Therefore, since the numbers $\mathrm{MG}(k,d,m,r)$ are upper bounded by a function belonging to the class $\mathcal{E}^9$ of Grzegorczyk's hierarchy and the numbers
  $\mathrm{MT}(d,m,r)$ are upper bounded by a function belonging to the class $\mathcal{E}^6$, by
  (ii) of Proposition \ref{prim_rec_properties}, we have that the numbers $\mathrm{MG}_{\pm}(k,d,m,r)$ are upper bounded by a function belonging to $\mathcal{E}^9$.

\section{Concluding Remarks}
It is easy to observe that $X_{\pm k}(n)$ does not admit the Ramsey property, just by considering a
coloring depending on the sign of the function at the minimum of its support.
It is natural then to ask whether $X_{\pm k}(n)$  admits a Ramsey degree (for details concerning the notion of Ramsey degree, we refer the reader to \cite{KPT}).
The answer turns out to be negative. In particular, the following holds.

\begin{prop}
  For every pair of positive integers $K,n$ with  $n\meg 2K$ there exists a coloring $c:X_{\pm 1}(n)\to \{1,....,K\}$ such that for every block sequence $\mathbf{F}$ in $X_{\pm k}(n)$ of length $2K$, we have that $\langle\mathbf{F}\rangle_{\pm k}$ realizes all the colors.
\end{prop}
\begin{proof}
  Let $K,n$ be positive integers with $2K\mik n$. We define a coloring $c:X_{\pm 1}(n)\to \{1,....,K\}$ by the rule
  \[c(f)=(|\mathrm{tp}(f)|\mod K)+1.\]
  The color essentially counts the ``jumps''  between the two non-zero values of the function $f$.
  Let $\mathbf{F}=(f_i)_{i=0}^{2K-1}$ be a block sequence in $X_{\pm k}(n)$.
  For every $i=0,...,K-1$, we set
  \[g_i=f_{2i}\big(\min\supp(f_{2i})\big)\cdot f_{2i}+f_{2i+1}\big(\max\supp(f_{2i+1})\big)\cdot f_{2i+1}.\]
  Let us observe that $(g_i)_{i=0}^{K-1}$ forms a block sequence of length $K$ and
  $g_i(\min\supp(g_i))=g_i(\max\supp(g_i))=1$, for all $0\mik i<K$.
  We set
  \[h_0=\sum_{j=0}^{K-1}g_j\;
  \text{and}\; h_i=\sum_{j=0}^{i-1}(-1)^jg_j+(-1)^i\sum_{j=i}^{K-1}g_j\]
  for all $1\mik i<K$. By this choice we have that $h_i\in \langle\mathbf{F}\rangle_{\pm k}$
  and
  \begin{equation}
    \label{eq22}
    |\mathrm{tp}(h_i)|=|\mathrm{tp}(h_0)|+i
  \end{equation}
  for all $0\mik i<K$. By \eqref{eq22} we have that the set $\{h_i:0\mik i<K\}$ itself
  realizes all the colors. The proof is complete.
\end{proof}


\end{document}